\documentclass[draft,twoside]{article}
\usepackage{amsmath,amsthm,amssymb,ifthen}

\numberwithin{equation}{section}

\AtBeginDocument{\setlength{\baselineskip}{13.0pt}}

\newtheorem{thm}{Theorem}[section]
\newtheorem{lem}[thm]{Lemma}
\newtheorem{cor}[thm]{Corollary}
\newtheorem{prop}[thm]{Proposition}
\newtheorem{conj}[thm]{Conjecture}

\theoremstyle{definition}
\newtheorem{example}[thm]{Example}

\newtheorem*{ackn}{Acknowledgement}

\theoremstyle{remark}
\newtheorem{remark}[thm]{Remark}

\newcommand{\eoe}{\hfill{$\diamondsuit$}}
\newenvironment{ex}{\begin{example}}{\eoe\end{example}}
\newenvironment{rem}{\begin{remark}}{\eoe\end{remark}}

\newcommand{\deq}{:=}
\newcommand{\eqsp}{\phantom{{}={}}}

\newcommand{\Zp}{\mathbb{Z}_{>0}}
\newcommand{\Znn}{\mathbb{Z}_{\ge0}}

\newcommand{\C}{\mathbb{C}}
\newcommand{\R}{\mathbb{R}}

\newcommand{\kakko}[1]{\left(#1\right)}
\newcommand{\ckakko}[1]{\left\{#1\right\}}
\newcommand{\floor}[1]{\left\lfloor#1\right\rfloor}

\newcommand{\Heun}{\mathcal{D}_{\text{\upshape\tiny \!H}}}

\newcommand{\Ochiai}{\mathcal{D}_{\text{\upshape\tiny \!O}}}
\newcommand{\Wakayama}{\mathcal{D}_{\text{\upshape\tiny \!W}}}

\newcommand{\J}[1]{J_{#1}}
\newcommand{\tJ}[1]{\tilde{J}_{#1}}
\newcommand{\w}[1]{w_{#1}}

\newcommand{\I}[3]{I^{(#1)}_{#2,#3}}
\newcommand{\sI}[3]{\mathbb{I}^{(#1)}_{#2,#3}}
\newcommand{\aI}[3]{\widetilde{\mathbb{I}}^{(#1)}_{#2,#3}}
\newcommand{\g}[1]{g_{#1}}
\newcommand{\tg}[1]{\widetilde{g}_{#1}}

\newcommand{\sK}[1]{K_{#1}}
\newcommand{\sP}[1]{M_{#1}}
\newcommand{\sM}[1]{M_{#1}}
\newcommand{\sL}[1]{L_{#1}}

\newcommand{\sQ}[1]{Q_{#1}}
\newcommand{\sS}[1]{S_{#1}}

\newcommand{\bJ}[1]{\mathbb{J}_{#1}}

\newcommand{\gA}[1][k]{\mathcal{A}^{(#1)}}
\newcommand{\gB}[1][k]{\mathcal{B}^{(#1)}}
\newcommand{\ga}[1][k]{a^{(#1)}}
\newcommand{\gb}[1][k]{b^{(#1)}}
\newcommand{\tgb}[1][k]{\tilde{b}^{(#1)}}

\newcommand{\dint}{\int_0^1\!\!\int_0^1}
\newcommand{\tint}{\int_0^1\!\!\int_0^1\!\!\int_0^1}

\newcommand{\mint}{\int_0^1\!\!\int_0^1\!\!\dotsb\int_0^1\!\!}
\newcommand{\Mint}{\int_0^\infty\!\!\int_0^\infty\!\!\dotsb\int_0^\infty\!\!}

\DeclareMathOperator{\Tr}{Tr} 
\DeclareMathOperator{\Spec}{Spec} 

\newcommand{\LS}[2]{\left(\!\frac{#1}{#2}\!\right)}
\newcommand{\phs}[2]{\left(#1\right)_{#2}}
\def\hgf#1#2(#3;#4;#5){{}_{#1}F_{#2}\!\left(#3;#4;#5\right)}

\newcommand{\apery}[1]{A_{#1}}
\newcommand{\bpery}[1]{B_{#1}}

\newcommand{\Apery}[1]{\mathcal{A}_{#1}}
\newcommand{\Bpery}[1]{\mathcal{B}_{#1}}
\newcommand{\Rpery}[1]{\mathcal{R}_{#1}}

\newcommand{\ascent}[1]{{#1}^{\sharp}}
\newcommand{\e}[1]{\varepsilon_{#1}}

\title{\bfseries Higher Ap\'ery-like numbers
arising from\\
special values of the spectral zeta function\\
for the non-commutative harmonic oscillator}
\author{Kazufumi Kimoto}
\date{January 20, 2009}

\begin{document}

\maketitle

\begin{abstract}
A generalization of the Ap\'ery-like numbers,
which is used to describe the special values $\zeta_Q(2)$ and $\zeta_Q(3)$
of the spectral zeta function for the non-commutative harmonic oscillator, are introduced and studied.
In fact, we give a recurrence relation for them, which shows a ladder structure among them.
Further, we consider the `rational part' of the higher Ap\'ery-like numbers.
We discuss several kinds of congruence relations among them,
which are regarded as an analogue of the ones among Ap\'ery numbers.
\end{abstract}

\section{Introduction}

The \emph{non-commutative harmonic oscillator} is the system of differential equations defined by the operator
\begin{align}
Q=Q_{\alpha,\beta}\deq
\begin{pmatrix}\alpha & 0 \\ 0 & \beta\end{pmatrix}\left(-\frac12\frac{d^2}{dx^2}+\frac12x^2\right)%
+\begin{pmatrix}0 & -1 \\ 1 & 0\end{pmatrix}\left(x\frac{d}{dx}+\frac12\right),
\end{align}
where $\alpha$ and $\beta$ are real parameters.
In this paper, we always assume that $\alpha>0$, $\beta>0$ and $\alpha\beta>1$.
Under these conditions, one can show that the operator $Q$ defines an unbounded, positive, self-adjoint operator
on the space $L^2(\R;\C^2)$ of $\C^2$-valued square integrable functions which has only a discrete spectrum,
and the multiplicities $m(\lambda)$ of the eigenvalues $\lambda\in\Spec(Q)$ are uniformly bounded \cite{PW2001}.
Hence, in this case, it is meaningful to define its \emph{spectral zeta function}
$\zeta_Q(s)=\Tr Q^{-s}=\sum_{\lambda\in\Spec(Q)}m(\lambda)\lambda^{-s}$.
This series converges absolutely if $\Re s>1$,
and hence defines a holomorphic function on the half plane $\Re s>1$.
Further, $\zeta_Q(s)$ is meromorphically continued to the whole complex plane $\C$ which has `trivial zeros'
at $s=0,-2,-4,\dots$ (see \cite{IW2005a}, \cite{P2007COELN}).

The aim of this paper is to study the \emph{higher Ap\'ery-like numbers} $\J k(n)$ defined by
\begin{align*}
\J k(n)\deq
2^k\int_{[0,1]^k}
\left(\frac{(1-x_1^4)(1-x_2^4\dotsb x_k^4)}{(1-x_1^2\dotsb x_k^2)^2}\right)^{\!n\!}
\frac{dx_1dx_2\dotsb dx_k}{1-x_1^2\dotsb x_k^2}
\end{align*}
for $k\ge2$ and $n\ge0$,
which are a generalization of the Ap\'ery-like numbers $\J2(n)$ and $\J3(n)$ studied in \cite{KW2006a}.
This object arises from the special values of the spectral zeta function $\zeta_Q(s)$:
In \cite{IW2005b},
the generating functions of the numbers $\J2(n)$ and $\J3(n)$ are
used to describe the special values $\zeta_Q(2)$ and $\zeta_Q(3)$ of the spectral zeta function $\zeta_Q(s)$.
Similarly,
the higher Ap\'ery-like numbers $\J k(n)$ are closely related to the special values $\zeta_Q(k)$ (see \S \ref{subsec:specialvalue}).

We first show that $\J k(n)$ satisfy three-term (inhomogeneous) recurrence relations,
which is translated to (inhomogeneous) singly confluent Heun differential equations for their generating functions.
The point is that these relations or differential equations are connecting $\J k(n)$'s and $\J{k-2}(n)$'s.
This fact implies that there could be a certain relation between $\zeta_Q(k)$ and $\zeta_Q(k-2)$.
It would be very interesting if one can utilize these relations to understand
a modular interpretation of $\zeta_Q(4), \zeta_Q(6), \dots$ based on that of $\zeta_Q(2)$ (see \cite{KW2006b}).
We also notice that these recurrence relations quite resemble to those for \emph{Ap\'ery numbers}
used to prove the irrationality of $\zeta(2)$ and $\zeta(3)$ (see \cite{Po1979}),
and this is why we call $\J k(n)$ the (higher) Ap\'ery-like numbers.

By a suitable change of variable in the differential equation,
we also obtain another kind of recurrence relations,
which allow us to define the \emph{rational part} of the higher Ap\'ery-like numbers
(or \emph{normalized higher Ap\'ery-like numbers}) $\tJ k(n)$.
In fact, each $\J k(n)$ is a linear combination
of the Riemann zeta values $\zeta(k), \zeta(k-2), \dots$ and the coefficients are given by $\tJ m(n)$'s.
Since there are various kind of congruence relations satisfied by Ap\'ery numbers (see, e.g. \cite{B1985}, \cite{B1987}, \cite{AO2000JRAM}),
it would be natural and interesting to find an analogue for our higher Ap\'ery-like numbers.
Actually, we give several congruence relations among $\tJ2(n)$ and $\tJ3(n)$ in \cite{KW2006b}.
We add such congruence relations among $\tJ k(n)$, and give some conjectural congruences.

\section{Ap\'ery numbers for $\zeta(2)$ and $\zeta(3)$}

As a quick reference for the readers,
we recall the definitions and several properties on the original Ap\'ery numbers.

\subsection{Ap\'ery numbers for $\zeta(2)$}

\emph{Ap\'ery numbers for $\zeta(2)$} are given by
\begin{align*}
\apery2(n)=\sum_{k=0}^n\binom nk^{\!2}\binom{n+k}k,\quad
\bpery2(n)=\sum_{k=0}^n\binom nk^{\!2}\binom{n+k}k\kakko{2\sum_{m=1}^n\frac{(-1)^{m-1}}{m^2}+\sum_{m=1}^k\frac{(-1)^{n+m-1}}{m^2\binom nm\binom{n+m}m}}.
\end{align*}
These numbers satisfy a recurrence relation of the same form
\begin{align}\label{eq:reccurence_of_A2}
n^2u(n)-(11n^2-11n+3)u(n-1)-(n-1)^2u(n-2)=0\quad(n\ge2)
\end{align}
with initial conditions $\apery2(0)=1, \apery2(1)=3$ and $\bpery2(0)=0, \bpery2(1)=5$.
The ratio $\bpery2(n)/\apery2(n)$ converges to $\zeta(2)$,
and this convergence is rapid enough to prove the irrationality of $\zeta(2)$.
Consider the generating functions
\begin{align*}
\Apery2(t)=\sum_{n=0}^\infty \apery2(n)t^n,\quad
\Bpery2(t)=\sum_{n=0}^\infty \bpery2(n)t^n,\quad
\Rpery2(t)=\Apery2(t)\zeta(2)-\Bpery2(t).
\end{align*}
It is proved that
\begin{align*}
L_2\Apery2(t)=0,\quad L_2\Bpery2(t)=-5,\quad L_2\Rpery2(t)=5,
\end{align*}
where $L_2$ is a differential operator given by
\begin{align*}
L_2=t(t^2+11t-1)\frac{d^2}{dt^2}+(3t^2+22t-1)\frac{d}{dt}+(t+3).
\end{align*}
The function $\Rpery2(t)$ is also expressed as follows:
\begin{align*}
\Rpery2(t)=\dint\frac{dxdy}{1-xy+txy(1-x)(1-y)}.
\end{align*}
The family $Q^2_t:1-xy+txy(1-x)(1-y)=0$ of algebraic curves,
which comes from the denominator of the integrand,
is birationally equivalent to the universal family $C^2_t$ of elliptic curves having rational $5$-torsion.
Moreover, the differential equation $L_2\Apery2(t)=0$ is regarded as a Picard-Fuchs equation for this family,
and $\Apery2(t)$ is interpreted as a period of $C^2_t$ (see \cite{B1983}).

\subsection{Ap\'ery numbers for $\zeta(3)$}

\emph{Ap\'ery numbers for $\zeta(3)$} are given by
\begin{align*}
\apery3(n)=\sum_{k=0}^n\binom nk^{\!2}\binom{n+k}k^{\!2},\quad
\bpery3(n)=\sum_{k=0}^n\binom nk^{\!2}\binom{n+k}k^{\!2}\kakko{\sum_{m=1}^n\frac1{m^3}+\sum_{m=1}^k\frac{(-1)^{m-1}}{2m^3\binom nm\binom{n+m}m}}
\end{align*}
These numbers satisfy a recurrence relation of the same form
\begin{align*}
n^3u(n)-(34n^3-51n^2+27n-5)u(n-1)+(n-1)^3u(n-2)=0\quad(n\ge2)
\end{align*}
with initial conditions $\apery3(0)=1, \apery3(1)=5$ and $\bpery3(0)=0, \bpery3(1)=6$.
The ratio $\bpery3(n)/\apery3(n)$ converges to $\zeta(3)$ rapidly enough to allow us to prove the irrationality of $\zeta(3)$.
Consider the generating functions
\begin{align*}
\Apery3(t)=\sum_{n=0}^\infty \apery3(n)t^n,\quad
\Bpery3(t)=\sum_{n=0}^\infty \bpery3(n)t^n,\quad
\Rpery3(t)=\Apery3(t)\zeta(3)-\Bpery3(t).
\end{align*}
It is proved that
\begin{align*}
L_3\Apery3(t)=0,\quad L_3\Bpery3(t)=5,\quad L_3\Rpery3(t)=-5,
\end{align*}
where $L_3$ is a differential operator given by
\begin{align*}
L_3=t^2(t^2-34t^2+1)\frac{d^3}{dt^3}+t(6t^2-153t+3)\frac{d^2}{dt^2}+(7t^2-112t+1)\frac{d}{dt}+(t-5).
\end{align*}
The function $\Rpery3(t)$ is also expressed as follows:
\begin{align*}
\Rpery3(t)=\tint\frac{dxdydz}{1-(1-xy)z-txyz(1-x)(1-y)(1-z)}.
\end{align*}
The family $Q^3_t:1-(1-xy)z-txyz(1-x)(1-y)(1-z)=0$ of algebraic surfaces
coming from the denominator of the integrand is birationally equivalent to
a certain family $C^3_t$ of $K3$ surfaces with Picard number $19$.
Furthermore, the differential equation $L_3\Apery3(t)=0$ is regarded as a Picard-Fuchs equation for this family,
and $\Apery3(t)$ is interpreted as a period of $C^3_t$ (see \cite{BP1984}).

\subsection{Congruence relations for Ap\'ery numbers}

Ap\'ery numbers $\apery2(n)$ and $\apery3(n)$ have various kind of congruence properties.
Here we pick up several of them, for which we will discuss an Ap\'ery-like analogue later.

\begin{prop}
Let $p$ be a prime and $n=n_0+n_1p+\dots+n_kp^k$ be the $p$-ary expansion of $n\in\Znn$ $(0\le n_j<p)$.
Then it holds that
\begin{align*}
\apery2(n)&\equiv\prod_{j=0}^k\apery2(n_j) \pmod p,\qquad
\apery3(n)\equiv\prod_{j=0}^k\apery3(n_j) \pmod p.
\end{align*}
\end{prop}

\begin{prop}[{\cite[Theorems 1 and 2]{B1985}}]\label{prop:Apery_congruence}
For all odd prime $p$, it holds that
\begin{align*}
\apery2(mp^r-1)&\equiv\apery3(mp^{r-1}-1)\pmod{p^r},\\
\apery3(mp^r-1)&\equiv\apery3(mp^{r-1}-1)\pmod{p^r}
\end{align*}
for any $m,r\in\Zp$.
These congruence relations hold modulo $p^{3r}$ if $p\ge5$
{\upshape(}known and referred to as a \emph{supercongruence}{\upshape)}.
\end{prop}

We denote by $\eta(\tau)$ the Dedekind eta function
\begin{equation}
\eta(\tau)=q^{1/24}\prod_{n=1}^\infty(1-q^n),\quad q=e^{2\pi i\tau}\quad(\Im\tau>0).
\end{equation}

\begin{prop}[{\cite[Theorem 13.1]{SB1985}}]
For any odd prime $p$ and any $m,r\in\Zp$ with $m$ odd, it holds that
\begin{align}\label{eq:ASD-type_congruence_for_A2}
\apery2(\tfrac{mp^r-1}2)-\lambda_p\apery2(\tfrac{mp^{r-1}-1}2)+(-1)^{(p-1)/2}p^2\apery2(\tfrac{mp^{r-2}-1}2)\equiv0 \pmod{p^r}.
\end{align}
Here $\lambda_n$ is defined by
\begin{align*}
\sum_{n=1}^\infty \lambda_nq^n=\eta(4\tau)^6=q\prod_{n=1}^\infty(1-q^{4n})^6.
\end{align*}
\end{prop}

\begin{prop}[{\cite[Theorem 4]{B1987}}]
For any odd prime $p$ and any $m,r\in\Zp$ with $m$ odd, it holds that
\begin{align}\label{eq:ASD-type_congruence_for_A3}
\apery3(\tfrac{mp^r-1}2)-\gamma_p\apery3(\tfrac{mp^{r-1}-1}2)+p^3\apery3(\tfrac{mp^{r-2}-1}2)\equiv0 \pmod{p^r}.
\end{align}
Here $\gamma_n$ is defined by
\begin{align*}
\sum_{n=1}^\infty \gamma_nq^n=\eta(2\tau)^4\eta(4\tau)^4=q\prod_{n=1}^\infty(1-q^{2n})^4(1-q^{4n})^4.
\end{align*}
\end{prop}

\section{Ap\'ery-like numbers for $\zeta_Q(2)$ and $\zeta_Q(3)$}

We introduce the Ap\'ery like numbers $\J2(n)$ and $\J3(n)$,
and give a brief explanation on their basic properties and the connection between
the special values $\zeta_Q(2), \zeta_Q(3)$ of the spectral zeta function $\zeta_Q(s)$.

\subsection{Definition}

We define the \emph{Ap\'ery-like numbers for $\zeta_Q(2)$ and $\zeta_Q(3)$} by
\begin{align*}
\J2(n)&\deq
4\dint\left(\frac{(1-x_1^4)(1-x_2^4)}{(1-x_1^2x_2^2)^2}\right)^{\!n\!}
\frac{dx_1dx_2}{1-x_1^2x_2^2},\\
\J3(n)&\deq
8\tint\left(\frac{(1-x_1^4)(1-x_2^4x_3^4)}{(1-x_1^2x_2^2x_3^2)^2}\right)^{\!n\!}
\frac{dx_1dx_2dx_3}{1-x_1^2x_2^2x_3^2}.
\end{align*}
The sequences $\{\J2(n)\}$ and $\{\J3(n)\}$ satisfy the recurrence formula (Propositions 4.11 and 6.4 in \cite{IW2005b})
\begin{align}
4n^2\J2(n)-(8n^2-8n+3)\J2(n-1)+4(n-1)^2\J2(n-2)&=0,\label{eq:rec_for_J2}\\
4n^2\J3(n)-(8n^2-8n+3)\J3(n-1)+4(n-1)^2\J3(n-2)
&=\frac{2^n(n-1)!}{(2n-1)!!}
\label{eq:rec_for_J3}
\end{align}
with the initial conditions
\begin{align*}
\J2(0)=3\zeta(2),\quad \J2(1)=\frac94\zeta(2);\qquad
\J3(0)=7\zeta(3),\quad \J3(1)=\frac{21}4\zeta(3)+\frac12.
\end{align*}
It is notable that the left-hand sides of these relations have the same shape.
Since the relations \eqref{eq:rec_for_J2},\eqref{eq:rec_for_J3}
and the one \eqref{eq:reccurence_of_A2} for $\apery2(n)$ have quite close shapes,
we call the numbers $\J2(n)$ and $\J3(n)$ the \emph{Ap\'ery-like} numbers.

\subsection{Generating functions and their differential equations}

The generating functions for $\J2(n)$ and $\J3(n)$ are defined by
\begin{align}
\w2(t)&\deq\sum_{n=0}^\infty \J2(n)t^n
=4\dint
\frac{1-x_1^2x_2^2}{(1-x_1^2x_2^2)^2%
-t(1-x_1^4)(1-x_2^4)}\,dx_1dx_2,\\
\w3(t)&\deq \sum_{n=0}^\infty \J3(n)t^n
=8\tint
\frac{1-x_1^2x_2^2x_3^2}{(1-x_1^2x_2^2x_3^2)^2%
-t(1-x_1^4)(1-x_2^4x_3^4)}\,dx_1dx_2dx_3.
\end{align}
By the recurrence relations \eqref{eq:rec_for_J2} and \eqref{eq:rec_for_J3},
we get the differential equations
\begin{align}
\Heun\w2(t)&=0,\label{eq:DE_for_w2}\\
\Heun\w3(t)&=\frac12\,\hgf21(1,1;\frac32;t),\label{eq:DE_for_w3}
\end{align}
where $\Heun$ denotes the singly confluent Heun differential operator given by
\begin{align}\label{eq:HeunOperator}
\Heun=t(1-t)^2\frac{d^2}{dt^2}+(1-3t)(1-t)\frac{d}{dt}+t-\frac34.
\end{align}
\eqref{eq:DE_for_w2} is solved in \cite{O} as
\begin{align*}
\w2(t)=\frac{3\zeta(2)}{1-t}\hgf21(\frac12,\frac12;1;\frac t{t-1}).
\end{align*}
Here $\hgf21(a,b;c;z)$ is the Gaussian hypergeometric function.
Now it is immediate that
\begin{align}\label{eq:formula_for_J2}
\J2(n)=3\zeta(2)\sum_{j=0}^n(-1)^j\binom{-\frac12}{j}^{\!2\!}\binom nj.
\end{align}
Similarly, \eqref{eq:DE_for_w3} is solved in \cite{KW2006a} as
\begin{align*}
\w3(t)&=\frac{7\zeta(3)}{1-t}\hgf21(\frac12,\frac12;1;\frac t{t-1})%
-2\sum_{n=0}^\infty
\left(
\sum_{k=0}^n(-1)^k\binom{-\frac12}{k}^{\!\!2}\binom{n}{k}
\sum_{j=0}^{k-1} \frac{1}{(2j+1)^3}\binom{-\frac12}{j}^{\!\!\!-2}
\right)t^n.
\end{align*}
Therefore it follows that
\begin{align}\label{eq:formula_for_J3}
\J3(n)&=7\zeta(3)\sum_{j=0}^n(-1)^j\binom{-\frac12}{j}^{\!\!2}\binom nj%
-2\sum_{j=0}^n(-1)^j\binom{-\frac12}{j}^{\!\!2}\binom nj
\sum_{k=0}^{j-1}\frac1{(2k+1)^3}\binom{-\frac12}{k}^{\!\!\!-2}.
\end{align}

\begin{rem}
The function
$$
W_2(T)=\hgf21(\frac12,\frac12;1;T^2)=\frac1{3\zeta(2)}(1-t)\w2(t)\qquad\kakko{T^2=\frac t{t-1}}
$$
satisfies the differential equation
\begin{align*}
\kakko{T(T^2-1)\frac{d^2}{dT^2}+(3T^2-1)\frac{d}{dT}+T}W_2(T)=0,
\end{align*}
which can be regarded as a Picard-Fuchs equation for the universal family of elliptic curves
having rational $4$-torsion \cite{KW2006b}.
This is an analogue of the result \cite{B1983} for the Ap\'ery numbers for $\zeta(2)$ (see also Section 2.1).
It is natural to ask whether there is such a modular interpretation for $\w3(t)$ (or ``$W_3(T)$'').
We have not obtained an answer to this question so far.
\end{rem}

\subsection{Connection to the special values of $\zeta_Q(s)$}\label{subsec:specialvalue}

We also introduce another kind of generating functions for $\J k(n)$ as
\begin{align*}
\g2(z)&\deq\sum_{n=0}^\infty \binom{-\frac12}n\J2(n)z^n
=4\dint\frac{dx_1dx_2}{\sqrt{(1-x_1^2x_2^2)^2+z(1-x_1^4)(1-x_2^4)}},\\
\g3(z)&\deq\sum_{n=0}^\infty \binom{-\frac12}n\J2(n)z^n
=8\tint\frac{dx_1dx_2dx_3}{\sqrt{(1-x_1^2x_2^2x_3^2)^2+z(1-x_1^4)(1-x_2^4x_3^4)}}.
\end{align*}
The special values of $\zeta_Q(s)$ at $s=2,3$ are
given as follows.
\begin{thm}[Ichinose-Wakayama \cite{IW2005b}]\label{thm:specialvaluesIW}
If $\alpha\beta>2$ {\upshape(}i.e. $0<1/(1-\alpha\beta)<1${\upshape)}, then
\begin{align*}
\zeta_Q(2)&=2\kakko{\frac{\alpha+\beta}{2\sqrt{\alpha\beta(\alpha\beta-1)}}}^{\!2}
\kakko{\zeta(2,\frac12)+\kakko{\frac{\alpha-\beta}{\alpha+\beta}}^{\!2}
\g2\!\kakko{\frac1{\alpha\beta-1}}},\\
\zeta_Q(3)&=2\kakko{\frac{\alpha+\beta}{2\sqrt{\alpha\beta(\alpha\beta-1)}}}^{\!3}
\kakko{\zeta(3,\frac12)+3\kakko{\frac{\alpha-\beta}{\alpha+\beta}}^{\!2}
\g3\!\kakko{\frac1{\alpha\beta-1}}},
\end{align*}
where $\zeta(s,x)=\sum_{n=0}^\infty(n+x)^{-s}$ is the Hurwitz zeta function.
\end{thm}

\begin{rem}
We can determine the functions $\g2(x)$ and $\g3(x)$ as follows:
\begin{align*}
\g2(x)=\J2(0)\tg2(x),\qquad
\g3(x)=\J3(0)\tg2(x)+\tg3(x),
\end{align*}
where
\begin{align*}
\tg2(x)&\deq\frac1{\sqrt{1+x}}\,\hgf21(\frac14,\frac14;1;\frac{x}{1+x})^{\!2}
=\hgf21(\frac14,\frac34;1;-x)^{\!2},\\
\tg3(x)&\deq\frac{-2}{\sqrt{1+x}}
\sum_{n=1}^\infty(-1)^n\binom{-\frac12}{n}^{\!\!3}\left(\frac{x}{1+x}\right)^{\!\!n}
\sum_{j=0}^{n-1}\frac1{(2j+1)^3}\binom{-\frac12}{j}^{\!\!\!-2}.
\end{align*}
See \cite{O} and \cite{KW2006a} for detailed calculation.
\end{rem}

\section{Higher Ap\'ery-like numbers}\label{sec:HAN}

Looking at the definition of $\J2(n)$ and $\J3(n)$,
it is natural to introduce the numbers $\J k(n)$ by
\begin{align*}
\J k(n)\deq
2^k\int_{[0,1]^k}
\left(\frac{(1-x_1^4)(1-x_2^4\dotsb x_k^4)}{(1-x_1^2\dotsb x_k^2)^2}\right)^{\!n\!}
\frac{dx_1dx_2\dotsb dx_k}{1-x_1^2\dotsb x_k^2}.
\end{align*}
We refer to $\J k(n)$ as \emph{higher Ap\'ery-like numbers}.
In fact, the generating function
\begin{align}
\g k(z)&\deq\sum_{n=0}^\infty\binom{-\frac12}{n}\J k(n)z^n
=2^k\int_{[0,1]^k}\frac{dx_1dx_2\dots dx_k}{\sqrt{(1-x_1^2x_2^2\dots x_k^2)^2+z(1-x_1^4)(1-x_2^4\dots x_k^4)}}
\end{align}
and its further generalizations are used to describe
the `higher' special values $\zeta_Q(k)$ ($k\ge4$) like Theorem \ref{thm:specialvaluesIW}
(see Remark \ref{rem:zeta_Q(4)} below).

It is immediate that $\J k(0)=(2^k-1)\zeta(k)$.
Further, as we mentioned in \cite{KW2006a},
the formula
\begin{align}\label{eq:J_k(1)}
\J k(1)=\frac34\sum_{m=0}^{\floor{k/2}-1}\frac1{4^m}\zeta\kakko{k-2m,\frac12}+\frac{1-(-1)^k}{2^{k-1}}
\end{align}
holds (see \S\ref{sec:J_k(1)} for the calculation).
It is directly verified that
\begin{align*}
4\J k(1)-3\J k(0)=\J{k-2}(1)\qquad(k\ge4).
\end{align*}

\begin{rem}\label{rem:zeta_Q(4)}
We can calculate that
\begin{multline*}
\zeta_Q(4)
=2\kakko{\frac{\alpha+\beta}{2\sqrt{\alpha\beta(\alpha\beta-1)}}}^{\!\!4}
\Biggl(\zeta(4,1/2)%
+4\kakko{\frac{\alpha-\beta}{\alpha+\beta}}^{\!\!2}\g4\!\kakko{\frac1{\alpha\beta-1}}\\
\eqsp+2\kakko{\frac{\alpha-\beta}{\alpha+\beta}}^{\!\!2}
\int_{[0,1]^4}\frac{16dx_1dx_2dx_3dx_4}{\sqrt{(1-x_1^2x_2^2x_3^2x_4^2)^2+\gamma_1(1-x_1^4x_2^4)(1-x_3^4x_4^4)}}\\
\eqsp+\kakko{\frac{\alpha-\beta}{\alpha+\beta}}^{\!\!4}
\int_{[0,1]^4}\frac{16dx_1dx_2dx_3dx_4}{\sqrt{(1-x_1^2x_2^2x_3^2x_4^2)^2+\gamma_1(1-x_1^4x_2^4)(1-x_3^4x_4^4)+\gamma_2(1-x_1^4)(1-x_2^4)(1-x_3^4)(1-x_4^4)}}
\Biggr),
\end{multline*}
where $\gamma_1=1/(\alpha\beta-1)$ and $\gamma_2=\alpha\beta/(\alpha\beta-1)^2$.
\end{rem}

Similar to the case of $\J2(n)$ and $\J3(n)$,
the higher Ap\'ery-like numbers $\J k(n)$ also satisfy a three-term recurrence relation as follows.
\begin{thm}\label{thm:recurrence_in_general}
The numbers $\J k(n)$ satisfy the recurrence relations
\begin{align}\label{eq:recurrence_of_han}
4n^2\J k(n)-(8n^2-8n+3)\J k(n-1)+4(n-1)^2\J k(n-2)=\J{k-2}(n)
\end{align}
for $n\ge2$ and $k\ge4$.
\end{thm}

We give the proof of Theorem \ref{thm:recurrence_in_general} in \S \ref{sec:proof_of_thm}.
It is remarkable that
the left-hand side of \eqref{eq:recurrence_of_han} has a common shape
with those of \eqref{eq:rec_for_J2} and \eqref{eq:rec_for_J3},
and \eqref{eq:recurrence_of_han} gives a `vertical' relation among $\J k(n)$'s,
i.e. it connects $\J k(n)$'s and $\J{k-2}(n)$'s.

\begin{ex}
First several terms of $\J4(n)$ are given by
\begin{align*}
\J4(0)=15\zeta(4),\quad
\J4(1)=\frac{45}4\zeta(4)+\frac9{16}\zeta(2),\quad
\J4(2)=\frac{615}{64}\zeta(4)+\frac{807}{1024}\zeta(2),\\
\J4(3)=\frac{2205}{256}\zeta(4)+\frac{3745}{4096}\zeta(2),\quad
\J4(4)=\frac{129735}{16384}\zeta(4)+\frac{1044135}{1048576}\zeta(2),\dots
\end{align*}
We also see that
\begin{align*}
4\J4(1)-3\J4(0)&=\frac94\zeta(2)=\J2(1),\\
16\J4(2)-19\J4(1)+4\J4(0)&=\frac{123}{64}\zeta(2)=\J2(2),\\
36\J4(3)-51\J4(2)+16\J4(1)&=\frac{441}{256}\zeta(2)=\J2(3),\\
64\J4(4)-99\J4(3)+36\J4(2)&=\frac{25947}{16384}\zeta(2)=\J2(4).
\end{align*}
\end{ex}

Define another kind of generating function for $\J k(n)$ by
\begin{align}
\w k(t)&\deq\sum_{n=0}^\infty \J k(n)t^n
=2^k\mint
\frac{1-x_1^2\dotsb x_k^2}{(1-x_1^2\dotsb x_k^2)^2%
-(1-x_1^4)(1-x_2^4\dotsb x_k^4)t}\,
dx_1dx_2\dotsb dx_k.
\end{align}
Theorem \ref{thm:recurrence_in_general} readily implies the
\begin{cor}\label{cor:DE_for_w}
The differential equation
\begin{align}\label{eq:DE_for_w}
\Heun\w k(t)=\frac{\w{k-2}(t)-\w{k-2}(0)}{4t}
\end{align}
holds for $k\ge4$.
Here $\Heun$ is the differential operator given in \eqref{eq:HeunOperator}.
\end{cor}

Put
\begin{align*}
\bJ0(n)\deq0,\quad
\bJ1(n)\deq\frac{(-1)^n}n,\quad
\bJ k(n)\deq\binom{-1/2}n\J k(n)\quad(k\ge2).
\end{align*}
By Theorem \ref{thm:recurrence_in_general}, we have
\begin{align*}
8n^3\bJ k(n)-(1-2n)(8n^2-8n+3)\bJ k(n-1)+2(n-1)(1-2n)(3-2n)\bJ k(n-2)=2n\bJ{k-2}(n)
\end{align*}
for $k\ge2$ and $n\ge1$.
Hence, if we put
\begin{align}
\Wakayama\deq8z^2(1+z)^2\frac{d^3}{dz^3}+24z(1+z)(1+2z)\frac{d^2}{dz^2}+2(4+27z+27z^2)\frac{d}{dz}+3(1+2z),
\end{align}
then we have the following (See also \cite[Proposition A.3]{KW2006a}).
\begin{cor}
The differential equations
\begin{align*}
\Wakayama\g2(z)&=0,\\
\Wakayama\g3(z)&=-\frac2{1+z},\\
\Wakayama\g k(z)&=2z\frac{d}{dz}\kakko{\frac{\g{k-2}(z)-\g{k-2}(0)}z}\quad(k\ge4)
\end{align*}
hold.
\qed
\end{cor}

\section{Proof of Theorem \ref{thm:recurrence_in_general}}\label{sec:proof_of_thm}

\subsection{Setting the stage}

Assume $k\ge2$.
We notice that
\begin{align*}
\J k(n)&=\Mint
\frac{e^{-(t_1+\dotsb+t_k)/2}(1-e^{-2t_1})^n(1-e^{-2(t_2+\dotsb+t_k)})^n}
{(1-e^{-(t_1+\dotsb+t_k)})^{2n+1}}\,dt_1\dotsb dt_k\\
&=\int_0^\infty\frac{e^{-u/2}}{(1-e^{-u})^{2n+1}}\,du
\int_0^u \frac{t^{k-2}}{(k-2)!}(1-e^{-2t})^n(1-e^{-2u+2t})^n\,dt
\end{align*}
for each $n\ge0$.
Let us introduce
\begin{align*}
\I knm=\I knm(u)\deq\int_0^u \frac{t^{k-2}}{(k-2)!}(1-e^{-2t})^n(1-e^{-2u+2t})^m\,dt
\end{align*}
for $n, m\ge0$.
We also put
\begin{align*}
\sI knm(u)\deq\frac12(\I knm(u)+\I kmn(u)),\qquad
\aI knm(u)\deq\frac12(\I knm(u)-\I kmn(u)).
\end{align*}
$\I knm(u)$ is symmetric in $n$ and $m$ if $k=2$ so that $\aI2nm(u)=0$,
but $\aI knm(u)\ne0$ in general.

It is convenient to set $\I knm(u)=0$ when $k<2$.
We see that
\begin{align*}
\J k(n)=\frac1{2^{2n+1}}
\int_0^\infty\frac{e^{nu}}{(\sinh\frac u2)^{2n+1}}\I{k}nn(u)\,du.
\end{align*}
Thus we also set $\J k(n)=0$ if $k<2$.
Under these convention, the following discussion for $\J k(n)$ is reduced to the one given by Ichinose and Wakayama \cite{IW2005b}
when $k=2,3$.

For later use, we define
\begin{align*}
\ga_{n}(u)&\deq \I knn(u)=\sI knn(u)\quad(n\ge0),\\
\gb_{n}(u)&\deq \frac12\kakko{\I kn{n-1}(u)+\I k{n-1}n(u)}=\sI kn{n-1}\quad(n\ge1),\\
\tgb_{n}(u)&\deq \frac12\kakko{\I kn{n-1}(u)-\I k{n-1}n(u)}=\aI kn{n-1}\quad(n\ge1),\\
\gA_{n}(u)&\deq e^{nu}\ga_{n}(u),\quad
\gB_{n}(u)\deq\frac{\gA_{n}(u)}{(\sinh\frac u2)^{2n+1}}\quad(n\ge0),
\end{align*}
so that
\begin{align*}
\J k(n)=\frac1{2^{2n+1}}
\int_0^\infty \gB_{n}(u)\,du.
\end{align*}

\subsection{Recurrence formulas for $\I knm(u)$}

Integration by parts implies
\begin{align}\label{eq:starting_formula}
\begin{split}
\I{k-1}nm&=\int_0^u \kakko{\frac d{dt}\frac{t^{k-2}}{(k-2)!}}
(1-e^{-2t})^n(1-e^{-2u+2t})^mdt\\
&=-\int_0^u \frac{t^{k-2}}{(k-2)!}\kakko{\frac d{dt}(1-e^{-2t})^n}(1-e^{-2u+2t})^mdt\\
&\eqsp-\int_0^u \frac{t^{k-2}}{(k-2)!}(1-e^{-2t})^{n}\kakko{\frac d{dt}(1-e^{-2u+2t})^m}dt.
\end{split}
\end{align}
when $n,m\ge1$.
Since
\begin{align*}
\frac d{dt}(1-e^{-2t})^n
&=2ne^{-2t}(1-e^{-2t})^{n-1}\\
&=2n\kakko{(1-e^{-2t})^{n-1}-(1-e^{-2t})^n},\\
\frac d{dt}(1-e^{-2u+2t})^m
&=-2me^{-2u+2t}(1-e^{-2u+2t})^{m-1}\\
&=-2m\kakko{(1-e^{-2u+2t})^{m-1}-(1-e^{-2u+2t})^{m}}
\end{align*}
for $n,m\ge1$,
we obtain the
\begin{lem}
The following three relations hold:
\begin{align}\label{eq:A}
\frac12\I{k-1}nm=(n-m)\I knm-n\I k{n-1}m+m\I kn{m-1}
\quad(n,m\ge1),
\end{align}
\begin{equation}\label{eq:B}
\begin{split}
n\I knm-(2n-1)\I k{n-1}m+(n-1)\I k{n-2}m-me^{-2u}\I k{n-1}{m-1}
=\frac12\kakko{\I{k-1}nm-\I{k-1}{n-1}m}
\quad(n\ge2, m\ge1),
\end{split}
\end{equation}
\begin{equation}\label{eq:B'}
\begin{split}
m\I knm-(2m-1)\I kn{m-1}+(m-1)\I kn{m-2}-ne^{-2u}\I k{n-1}{m-1}
=\frac12\kakko{\I{k-1}n{m-1}-\I{k-1}nm}
\quad(n\ge1, m\ge2).
\end{split}
\end{equation}
\qed
\end{lem}

Plugging \eqref{eq:A} into \eqref{eq:B}, we get
\begin{align}\label{eq:AB}
\I knm-(\I k{n-1}m+\I kn{m-1})+(1-e^{-2u})\I k{n-1}{m-1}=0
\quad(n\ge1,m\ge1),
\end{align}
which is a generalization of (4.14) in \cite{IW2005b}.
In particular, if we let $n=m$ in \eqref{eq:AB}, then we have
\begin{align}\label{eq:sono1}
\I knn-2\sI kn{n-1}+(1-e^{-2u})\I k{n-1}{n-1}=0.
\end{align}
Letting $m=n-1$ (or $n=m-1$ and exchanging $m$ by $n$) in \eqref{eq:AB},
we also have another specialization
\begin{align*}
\I kn{n-1}-(\I k{n-1}{n-1}+\I kn{n-2})+(1-e^{-2u})\I k{n-1}{n-2}=0
\quad(n\ge2),\\
\I k{n-1}n-(\I k{n-2}n+\I k{n-1}{n-1})+(1-e^{-2u})\I k{n-2}{n-1}=0
\quad(n\ge2).
\end{align*}
Adding these equations, we get
\begin{align}\label{eq:2step-descend}
\sI kn{n-2}=\gb_n(u)-\ga_{n-1}(u)+(1-e^{-2u})\gb_{n-1}(u)
\quad(n\ge2).
\end{align}

By specializing $m=n$ in \eqref{eq:B} and \eqref{eq:B'}, we have
\begin{align}
n\I knn-(2n-1)\I k{n-1}n+(n-1)\I k{n-2}n-ne^{-2u}\I k{n-1}{n-1}
=\frac12(\I{k-1}nn-\I{k-1}{n-1}n),\\
n\I knn-(2n-1)\I kn{n-1}+(n-1)\I kn{n-2}-ne^{-2u}\I k{n-1}{n-1}
=\frac12(\I{k-1}n{n-1}-\I{k-1}nn)
\end{align}
for $n\ge2$.
Similarly, specializing $m=n-1$ in \eqref{eq:B} and
$n=m-1$ in \eqref{eq:B'} (and exchanging $m$ by $n$),
we have
\begin{align*}
n\I kn{n-1}-(2n-1)\I k{n-1}{n-1}+(n-1)\I k{n-2}{n-1}-(n-1)e^{-2u}\I k{n-1}{n-2}
=\frac12(\I{k-1}n{n-1}-\I{k-1}{n-1}{n-1}),\\
n\I k{n-1}n-(2n-1)\I k{n-1}{n-1}+(n-1)\I k{n-1}{n-2}-(n-1)e^{-2u}\I k{n-2}{n-1}
=\frac12(\I{k-1}{n-1}{n-1}-\I{k-1}{n-1}n)
\end{align*}
for $n\ge2$.
Adding each pair of relations, we obtain
\begin{align}
2n\I knn-2(2n-1)\sI kn{n-1}+2(n-1)\sI kn{n-2}-2ne^{-2u}\I k{n-1}{n-1}=\aI{k-1}n{n-1},
\label{eq:sono2}\\
2n\sI kn{n-1}-2(2n-1)\I k{n-1}{n-1}+2(n-1)(1-e^{-2u})\sI k{n-1}{n-2}
=\aI{k-1}n{n-1}.\label{eq:sono3}
\end{align}

The formulas \eqref{eq:sono1}, \eqref{eq:sono2} and \eqref{eq:sono3}
are rewritten as follows.
\begin{lem}
The equations
\begin{align}
&\ga_n(u)+(1-e^{-2u})\ga_{n-1}(u)=2\gb_n(u),
\label{eq:b_by_a}\\
&n\ga_n(u)-(2n-1)\gb_n(u)+(n-1)\sI kn{n-2}-ne^{-2u}\ga_{n-1}(u)=\frac12\tgb[k-1]_n(u),
\label{eq:canceling_sIknn-2}\\
&n\gb_n(u)-(2n-1)\ga_{n-1}(u)+(n-1)(1-e^{-2u})\gb_{n-1}(u)=\frac12\tgb[k-1]_n(u)
\label{eq:canceling_sIknn-2b}
\end{align}
hold. \qed
\end{lem}

As a corollary, we also get
\begin{lem}
The equation
\begin{align}\label{eq:R}
n\ga_n(u)-(2n-1)(1+e^{-2u})\ga_{n-1}(u)+(n-1)(1-e^{-2u})^2\ga_{n-2}(u)=\tgb[k-1]_n(u)
\end{align}
holds.
\end{lem}

\begin{proof}
If we substitute \eqref{eq:b_by_a}, then we have
\begin{align*}
\tgb[k-1]_n(u)
&=2n\gb_n(u)-2(2n-1)\ga_{n-1}(u)+2(n-1)(1-e^{-2u})\gb_{n-1}(u)\\
&=n\kakko{\ga_n(u)+(1-e^{-2u})\ga_{n-1}(u)}-2(2n-1)\ga_{n-1}(u)\\
&\eqsp+(n-1)(1-e^{-2u})\kakko{\ga_{n-1}(u)+(1-e^{-2u})\ga_{n-2}(u)}\\
&=n\ga_n(u)-(2n-1)(1+e^{-2u})\ga_{n-1}(u)+(n-1)(1-e^{-2u})^2\ga_{n-2}(u),
\end{align*}
which is the desired formula.
\end{proof}

Here we give one more useful relation.
Using \eqref{eq:A} twice, we see that
\begin{align*}
\frac14\I{k-2}nn
&=\frac12\kakko{-n\I k{n-1}n+n\I kn{n-1}}\\
&=-n\kakko{-\I k{n-1}n-(n-1)\I k{n-2}n+n\I k{n-1}{n-1}}
+n\kakko{\I kn{n-1}-n\I k{n-1}{n-1}+(n-1)\I kn{n-2}}\\
&=n\kakko{2\gb_n(u)-2n\ga_{n-1}(u)+2(n-1)\sI kn{n-2}}.
\end{align*}
Thus we have
\begin{align}\label{eq:a[k-2]n-preformula}
\ga[k-2]_n(u)=8n\kakko{\gb_n(u)-n\ga_{n-1}(u)+(n-1)\sI kn{n-2}}.
\end{align}
Combining \eqref{eq:2step-descend}, \eqref{eq:a[k-2]n-preformula}
and \eqref{eq:canceling_sIknn-2b},
we obtain
\begin{lem}
The equation
\begin{align}\label{eq:tgb_by_a}
\ga[k-2]_n(u)=4n\tgb[k-1]_n(u)
\end{align}
holds. \qed
\end{lem}
In particular, the formula \eqref{eq:R} is rewritten as
\begin{equation}\label{eq:R'}
\begin{split}
n\ga_n(u)-(2n-1)(1+e^{-2u})\ga_{n-1}(u)+(n-1)(1-e^{-2u})^2\ga_{n-2}(u)=\frac1{4n}\ga[k-2]_n(u).
\end{split}
\end{equation}

\subsection{Relations for $\gB_n(u)$}

In view of \eqref{eq:starting_formula}, the differential
\begin{align*}
\frac d{du}\ga_n(u)
=2n\int_0^u\frac{t^{k-2}}{(k-2)!}(1-e^{-2t})^n e^{-2u+2t}(1-e^{-2u+2t})^{n-1}dt
\end{align*}
is written in two ways as
\begin{align*}
\frac d{du}\ga_n(u)
=2n\kakko{\I kn{n-1}-\I knn}
=-2n\kakko{\I knn-\I k{n-1}n}+\I{k-1}nn
\end{align*}
for $n\ge1$.
Hence it follows that
\begin{align}
\begin{split}
\frac d{du}\ga_n(u)
&=n\kakko{\I kn{n-1}-\I knn}-n\kakko{\I knn-\I k{n-1}n}+\frac12\I{k-1}nn\\
&=-n\ga_n(u)+n(1-e^{-2u})\ga_{n-1}(u)+\frac12\ga[k-1]_n(u).
\end{split}
\end{align}
Using this formula, we have
\begin{equation}\label{eq:DDEq_for_A}
\begin{split}
\frac d{du}\gA_n(u)-2n\sinh u\,\gA_{n-1}(u)
&=e^{nu}\kakko{\frac d{du}\ga_n(u)+n\ga_n(u)-n(1-e^{-2u})\ga_{n-1}(u)}\\
&=\frac12\gA[k-1]_n(u).
\end{split}
\end{equation}
Thus we obtain the
\begin{lem}
The equation
\begin{align}\label{eq:B-I}
2\tanh\frac u2\frac d{du}\gB_n(u)
=8n\gB_{n-1}(u)-(2n+1)\gB_n(u)+\tanh\frac u2\,\gB[k-1]_n(u)
\end{align}
holds for $n\ge1$.
\qed
\end{lem}

\begin{rem}
The differential of $\ga_0(u)$ is given by
\begin{align*}
\frac d{du}\ga_0(u)
=\frac{u^{k-2}}{(k-2)!}
\end{align*}
when $k\ge2$. If $k\ge3$, this is equal to $\ga[k-1]_0(u)$.
\end{rem}

We also see from \eqref{eq:R'} that
\begin{equation}\label{eq:B-O}
\begin{split}
&\eqsp n\gA_n(u)-2(2n-1)\cosh u\gA_{n-1}(u)+4(n-1)\sinh^2u\gA_{n-2}(u)\\
&=e^{nu}\tgb[k-1]_n(u)=\frac1{4n}\gA[k-2]_n(u).
\end{split}
\end{equation}
This implies the
\begin{lem}
The equation
\begin{align}\label{eq:B-II}
\begin{split}
n\kakko{1-\frac1{\cosh^2\frac u2}}\gB_n(u)
&=4(2n-1)\gB_{n-1}(u)-\frac{2(2n-1)}{\cosh^2\frac u2}\gB_{n-1}(u)\\
&\eqsp-16(n-1)\gB_{n-2}(u)+\frac1{4n}\kakko{1-\frac1{\cosh^2\frac u2}}\gB[k-2]_n(u)
\end{split}
\end{align}
holds for $n\ge2$.
\qed
\end{lem}

\subsection{Recurrence formula for $\J k(n)$}

Define
\begin{align}
\sK k(n)=\frac1{2^{2n+1}}\int_0^\infty \frac{\gB_n(u)}{\cosh^2\frac u2}du,\quad
\sP k(n)=\frac1{2^{2n+1}}\int_0^\infty \tanh\frac u2\,\gB[k-1]_n(u)\,du.
\end{align}
By integrating \eqref{eq:B-I} and \eqref{eq:B-II}, we have
\begin{align}
\sK k(n)&=(2n+1)\J k(n)-2n\J k(n-1)-\sP k(n),\\
\begin{split}
2n(\J k(n)-\sK k(n))&=(2n-1)(2\J k(n-1)-\sK k(n-1))-2(n-1)\J k(n-2)\\
&\eqsp+\frac1{2n}(\J{k-2}(n)-\sK{k-2}(n)).
\end{split}
\end{align}
Plugging these equations, we obtain
\begin{lem}
Put
\begin{align}
\sL k(n)\deq
\J{k-2}(n)-\J{k-2}(n-1)+2n\sP k(n)-(2n-1)\sP k(n-1)-\frac1{2n}\sP{k-2}(n).
\end{align}
The recurrence formula
\begin{align}\label{eq:general_recurrence}
4n^2\J k(n)-(8n^2-8n+3)\J k(n-1)+4(n-1)^2\J k(n-2)=\sL k(n)
\end{align}
holds for $k\ge2$ and $n\ge2$.
\qed
\end{lem}

When $k=2$, the inhomogeneous term $\sL2(n)$ in \eqref{eq:general_recurrence} vanishes
and we get \eqref{eq:rec_for_J2}.
When $k=3$, we see that
$\sL3(n)=2n\sP3(n)-(2n-1)\sP3(n-1)$, which is equal to
$\frac{2^n(n-1)!}{(2n-1)!!}$ (Lemma 6.3 in \cite{IW2005b}),
so we have \eqref{eq:rec_for_J3}.

\subsection{Calculation of the inhomogeneous terms}

Let us put
\begin{align}
\sQ k(n)\deq \frac1{2^{2n+1}}\int_0^\infty \frac{\gB_n(u)}{\tanh\frac u2}du.
\end{align}
This definite integral converges if $k\ge3$.

From \eqref{eq:B-I}, we have
\begin{align*}
2\frac d{du}\gB_n(u)
=8n\frac{\gB_{n-1}(u)}{\tanh\frac u2}%
-(2n+1)\frac{\gB_n(u)}{\tanh\frac u2}+\gB[k-1]_n(u).
\end{align*}
It follows then
\begin{align*}
0=8n\cdot 2^{2n-1}\sQ k(n-1)-(2n+1)2^{2n+1}\sQ k(n)+2^{2n+1}\J{k-1}(n),
\end{align*}
and hence
\begin{align}\label{eq:X-I}
\J{k-1}(n)=(2n+1)\sQ k(n)-2n\sQ k(n-1)
\end{align}
for $k\ge3$ and $n\ge1$.

From \eqref{eq:B-O}, we also see that
\begin{align*}
n\tanh\frac u2\,\gB_n(u)-2(2n-1)\kakko{\frac1{\tanh\frac u2}%
+\tanh\frac u2}\gB_{n-1}(u)%
+16(n-1)\frac{\gB_{n-2}(u)}{\tanh\frac u2}
=\frac1{4n}\tanh\frac u2\,\gB[k-2]_n(u).
\end{align*}
Thus we have
\begin{equation*}
\begin{split}
n 2^{2n+1}\sM{k+1}(n)-2(2n-1)2^{2n-1}\kakko{\sQ k(n-1)+\sM{k+1}(n-1)}\\
{}+16(n-1)2^{2n-3}\sQ k(n-2)=\frac1{4n}2^{2n+1}\sM{k-1}(n),
\end{split}
\end{equation*}
which implies
\begin{equation}\label{eq:X-II}
\begin{split}
2n\sM{k+1}(n)-(2n-1)\sM{k+1}(n-1)-\frac1{2n}\sM{k-1}(n)\\
=(2n-1)\sQ k(n-1)-2(n-1)\sQ k(n-2)
\end{split}
\end{equation}
for $k\ge3$ and $n\ge2$.

Using \eqref{eq:X-I} and \eqref{eq:X-II}, we obtain
\begin{align*}
&\eqsp2n\sM k(n)-(2n-1)\sM k(n-1)-\frac1{2n}\sM{k-2}(n)\\
&=(2n-1)\sQ{k-1}(n-1)-2(n-1)\sQ{k-1}(n-2)=\J{k-2}(n-1)
\end{align*}
for $k\ge4$ and $n\ge2$.
Hence the inhomogeneous term is computed as
\begin{align}
\sL k(n)=\J{k-2}(n)-\J{k-2}(n-1)+\J{k-2}(n-1)=\J{k-2}(n)
\end{align}
for $k\ge4$ and $n\ge2$.
This completes the proof of Theorem \ref{thm:recurrence_in_general}.

\begin{rem}
It may be ``natural'' to assume (or interpret) that
\begin{align*}
\J0(n)=0,\qquad
\J1(n)=2\int_0^1(1-x^2)^{n-1}dx=\frac{2^n(n-1)!}{(2n-1)!!}
\end{align*}
and
\begin{align*}
w_0(t)=0,\qquad
\text{``}w_1(t)-w_1(0)\text{''}=\sum_{n=1}^\infty\frac{2^n(n-1)!}{(2n-1)!!}t^n=2t\,\hgf21(1,1;\frac32;t).
\end{align*}
Under this convention,
Theorem \ref{thm:recurrence_in_general} and Corollary \ref{cor:DE_for_w} would include the case where $k=2,3$.
\end{rem}

\section{Infinite series expression}

We give an infinite series expression of $\J k(n)$.
Using it, we prove the equation \eqref{eq:J_k(1)}.

\subsection{Infinite series expression of $\J k(n)$}

Let us put
\begin{align*}
f_n(s,t)\deq\frac1{(1-s^2t^2)}\kakko{\frac{(1-s^4)(1-t^4)}{(1-s^2t^2)^2}}^{\!n\!}
=(1-s^4)^n(1-t^4)^n(1-s^2t^2)^{-2n-1}.
\end{align*}
Then we have
\begin{align*}
\J k(n)=2^k\mint f_n(x_1,x_2\dotsb x_k)dx_1\dotsb dx_k.
\end{align*}
Since
\begin{align*}
\begin{split}
f_n(s,t)
&=(1-s^4)^n(1-t^4)^n\sum_{l=0}^\infty\binom{-2n-1}l(-s^2t^2)^l
=\frac1{(2n)!}\sum_{l=0}^\infty \phs{l+1}{2n}s^{2l}(1-s^4)^nt^{2l}(1-t^4)^n,
\end{split}
\end{align*}
it follows that
\begin{align*}
\J k(n)=\frac{2^k}{(2n)!}\sum_{l=0}^\infty \phs{l+1}{2n}
I_1(l,n)I_{k-1}(l,n).
\end{align*}
Here $I_p(l,n)$ is given by
\begin{align*}
I_p(l,n)\deq\mint\
(u_1\dotsb u_p)^{2l}(1-(u_1\dotsb u_p)^4)^n\,du_1\dotsb du_p.
\end{align*}
Notice that
\begin{align*}
I_1(l,n)&=\int_0^1 u^{2l}(1-u^4)^ndu
=\frac{4^nn!}{(2l+1)(2l+5)\dotsb(2l+4n+1)},\\
I_p(l,n)
&=\sum_{j=0}^n(-1)^j\binom nj
\mint(u_1\dotsb u_p)^{2l+4j}du_1\dotsb du_p
=\sum_{j=0}^n(-1)^j\binom nj\frac1{(2l+4j+1)^p}.
\end{align*}
Thus we obtain the expression
\begin{align}
\J k(n)&=\frac{2^k4^nn!}{(2n)!}\sum_{l=0}^\infty
\frac{\phs{l+1}{2n}}{(2l+1)(2l+5)\dotsb(2l+4n+1)}
\sum_{j=0}^n\binom nj\frac{(-1)^j}{(2l+4j+1)^{k-1}}.
\end{align}

\subsection{Example: calculation of $\J k(1)$}\label{sec:J_k(1)}

When $n=1$, we see that
\begin{align*}
\begin{split}
\J k(1)&=2\cdot2^k\sum_{l=0}^\infty \frac{(l+1)(l+2)}{(2l+1)(2l+5)}
\kakko{\frac1{(2l+1)^{k-1}}-\frac1{(2l+5)^{k-1}}}\\
&=2\cdot2^k\sum_{l=0}^\infty
\frac{(l+1)(l+2)\kakko{(2l+5)^{k-1}-(2l+1)^{k-1}}}{(2l+1)^k(2l+5)^k}.
\end{split}
\end{align*}
Using the identity
\begin{align*}
&\eqsp(l+1)(l+2)\kakko{(2l+5)^{k-1}-(2l+1)^{k-1}}\\
&=\kakko{(2l+1)(2l+5)-(2l+1)+(2l+5)-1}\sum_{j=0}^{k-2}(2l+1)^j(2l+5)^{k-2-j},
\end{align*}
we have
\begin{align*}
\J k(1)&=2\cdot2^k\sum_{j=0}^{k-2}\Bigl\{
S(k-j-1,j+1)-S(k-j-1,j+2)+S(k-j,j+1)-S(k-j,j+2))\Bigr\},
\end{align*}
where
\begin{align*}
S(\alpha,\beta)\deq\sum_{l=0}^\infty (2l+1)^{-\alpha}(2l+5)^{-\beta}.
\end{align*}
Since
\begin{align*}
\sum_{j=1}^{k-1}S(j,k-j)
&=\sum_{l=0}^\infty\sum_{j=1}^{k-1}(2l+1)^{-j}(2l+5)^{j-k}\\
&=\sum_{l=0}^\infty\frac1{(2l+5)^k}\,\frac{2l+5}{2l+1}\,
\frac{1-\kakko{\frac{2l+5}{2l+1}}^{\!k-1\!}}{1-\kakko{\frac{2l+5}{2l+1}}}\\
&=\frac14\sum_{l=0}^\infty\kakko{\frac1{(2l+1)^{k-1}}-\frac1{(2l+5)^{k-1}}}
=\frac{1+3^{1-k}}{4},
\end{align*}
we have
\begin{align}\label{eq:Jk(1)_into_5parts}
\J k(1)=2^{k+1}\kakko{\frac2{3^{k+1}}+S(k,1)-S(1,k)+S(k+1,1)+S(1,k+1)}.
\end{align}
Let us calculate $S(k,1)$ and $S(1,k)$.
By the partial fraction expansion
\begin{align*}
\frac1{x(x+\alpha)^k}=\frac1{\alpha^k}\kakko{\frac1x-\frac1{x+\alpha}}%
-\sum_{m=2}^{k}\frac1{\alpha^{k-m+1}(x+\alpha)^m},
\end{align*}
we see that
\begin{align*}
\frac1{(2l+1)^k(2l+5)}
&=-\kakko{-\frac14}^{\!k\!}\kakko{\frac1{2l+1}-\frac1{2l+5}}%
+\frac1{2^k}\sum_{m=2}^k \kakko{-\frac12}^{\!k-m+2\!}\frac1{(l+\frac12)^m},\\
\frac1{(2l+1)(2l+5)^k}
&=\kakko{\frac14}^{\!k\!}\kakko{\frac1{2l+1}-\frac1{2l+5}}%
+\frac1{2^k}\sum_{m=2}^k \kakko{\frac12}^{\!k-m+2\!}\frac1{(l+2+\frac12)^m}.
\end{align*}
Thus it follows that
\begin{align*}
S(k,1)&=\frac1{2^k}
\sum_{m=2}^{k}\kakko{-\frac12}^{\!k-m+2\!}\zeta(m,\frac12)%
+\frac13\kakko{-\frac14}^{\!k-1\!},\\
S(1,k)&=-\frac1{2^k}
\sum_{m=2}^{k}\kakko{\frac12}^{\!k-m+2\!}\zeta(m,\frac12)%
+\frac13\kakko{\frac14}^{\!k-1\!}+\frac13+\frac1{3^k}.
\end{align*}
If we substitute these to \eqref{eq:Jk(1)_into_5parts},
then we have
\begin{align*}
\J k(1)&=2^{k+1}\Biggl(
\frac2{3^{k+1}}%
+\frac1{2^k}
\sum_{m=2}^{k}\kakko{-\frac12}^{\!k-m+2\!}\zeta(m,\frac12)%
+\frac13\kakko{-\frac14}^{\!k-1\!}\\
&\eqsp+\frac1{2^k}
\sum_{m=2}^{k}\kakko{\frac12}^{\!k-m+2\!}\zeta(m,\frac12)%
+\frac13\kakko{\frac14}^{\!k-1\!}-\frac13-\frac1{3^k}\\
&\eqsp+\frac1{2^{k+1}}
\sum_{m=2}^{k+1}\kakko{-\frac12}^{\!k-m+3\!}\zeta(m,\frac12)%
+\frac13\kakko{-\frac14}^{\!k\!}\\
&\eqsp-\frac1{2^{k+1}}
\sum_{m=2}^{k+1}\kakko{\frac12}^{\!k-m+3\!}\zeta(m,\frac12)%
-\frac13\kakko{\frac14}^{\!k\!}+\frac13+\frac1{3^{k+1}}\Biggr).
\end{align*}
Now it is straightforward to see that
\begin{align*}
\J k(1)
&=3\sum_{m=2}^{k}\frac{1+(-1)^{k-m}}{2^{k-m+3}}\zeta(m,\frac12)%
+\frac{1+(-1)^{k-1}}{2^{k-1}}\\
&=\frac34\sum_{\substack{2\le m\le k\\ 2\,|\,k-m}}
\frac{2^m-1}{2^{k-m}}\zeta(m)+\frac{1+(-1)^{k-1}}{2^{k-1}}\\
&=\frac34\sum_{m=0}^{\floor{k/2}-1}
2^{-2m}\zeta\kakko{k-2m,\frac12}+\frac{1-(-1)^k}{2^{k-1}}.
\end{align*}

\section{Differential equations for generating functions}

Utilizing the differential equations for the generating functions $\w k(t)$,
we give another kind of relations among the higher Ap\'ery-like numbers $\J k(n)$.

\subsection{Equivalent differential equations}

Consider the inhomogeneous (singly confluent) Heun differential equation
\begin{align*}
\Heun w(t)=u(t)
\end{align*}
for a given function $u(t)$.
Put $z=\frac t{t-1}$ and $v(z)=(1-t)w(t)$.
Then we have
\begin{align*}
\Ochiai v(z)=\frac1{z-1}\,u\!\kakko{\frac z{z-1}}.
\end{align*}
Here $\Ochiai$ is the \emph{hypergeometric} differential operator given by
\begin{align*}
\Ochiai=z(1-z)\frac{d^2}{dz^2}+(1-2z)\frac d{dz}-\frac14.
\end{align*}
We also remark that this is also the Picard-Fuchs differential operator
for the family $y^2=x(x-1)(x-z)$ of elliptic curves.

\subsection{Recurrence formula for $\J k(n)$}

Put $z=\frac t{t-1}$ and $v_k(z)=(1-t)w_k(t)$.
By Theorem \ref{thm:recurrence_in_general}, $v_k(z)$ satisfies the
differential equation
\begin{equation}\label{eq:DE_for_vk(z)}
\begin{split}
(\Ochiai v)(z)
&=\frac1{4(z-1)}\sum_{j=0}^\infty \J{k-2}(j+1)\kakko{\frac z{z-1}}^{\!j\!}\\
&=\sum_{n=0}^\infty\kakko{\frac14\sum_{j=0}^n(-1)^{j+1}\binom nj\J{k-2}(j+1)}z^n.
\end{split}
\end{equation}
The polynomial functions
\begin{align}
p_n(z)\deq-\frac4{(2n+1)^2}\binom{-\frac12}n^{\!-2\!}
\sum_{k=0}^n\binom{-\frac12}k^{\!2\!}z^k
\end{align}
satisfy the equation
\begin{align}
(\Ochiai p_n)(z)=z^n.
\end{align}
Hence we can construct a local holomorphic solution
to \eqref{eq:DE_for_vk(z)} as
\begin{align}
v(z)=
\sum_{n=0}^\infty\kakko{\frac14\sum_{j=0}^n(-1)^{j+1}\binom nj\J{k-2}(j+1)}p_n(z).
\end{align}
Notice that the difference $v_k(z)-v(z)$ satisfies the homogeneous
differential equation
\begin{align}
(\Ochiai(v_k-v))(z)=0.
\end{align}
Thus it follows that
\begin{align}
v_k(z)-v(z)=C_kv_2(z),
\end{align}
where the constant $C_k$ is determined by
\begin{align}
C_k=\frac{v_k(0)-v(0)}{v_2(0)}=\frac{(2^k-1)\zeta(k)-v(0)}{3\zeta(2)},
\end{align}
and $v(0)$ is given by
\begin{align}
v(0)=-\sum_{n=0}^\infty\frac1{(2n+1)^2}\binom{-\frac12}n^{\!-2\!}
\sum_{j=0}^n(-1)^{j+1}\binom nj\J{k-2}(j+1).
\end{align}
Therefore we have
\begin{align}
v_k(z)=(2^k-1)\zeta(k)\hgf21(\frac12,\frac12;1;z)%
+\kakko{v(z)-v(0)\hgf21(\frac12,\frac12;1;z)}.
\end{align}
Consequently, we obtain the
\begin{thm}\label{thm:inductive_expression_for_Jk}
When $k\ge4$, the equation
\begin{equation}\label{eq:inductive_expression_for_Jk}
\begin{split}
\J k(n)&=\sum_{p=0}^n (-1)^p\binom{-\frac12}p^{\!\!2}\binom np%
\kakko{(2^k-1)\zeta(k)-\sum_{i=0}^{p-1}\frac1{(2i+1)^2}\binom{-\frac12}{i}^{\!\!-2}
\sum_{j=0}^i(-1)^j\binom ij\J{k-2}(j+1)}
\end{split}
\end{equation}
holds.
\qed
\end{thm}

\begin{rem}
If we \emph{formally} put $\J1(n)=\frac{2^n(n-1)!}{(2n-1)!!}$ in \eqref{eq:inductive_expression_for_Jk}, then we have
\begin{align*}
\J3(n)=\sum_{p=0}^n (-1)^p\binom{-\frac12}p^{\!\!2}\binom np
\kakko{7\zeta(3)-2\sum_{i=0}^{p-1}\frac1{(2i+1)^3}\binom{-\frac12}{i}^{\!\!-2}}
\end{align*}
since
\begin{align*}
\sum_{j=0}^i(-1)^j\binom ij\frac{2^{j+1}j!}{(2j+1)!!}
=\frac2{2i+1}.
\end{align*}
This is nothing but the explicit formula \eqref{eq:formula_for_J3} for $\J3(n)$.
\end{rem}

\begin{ex}
Since
\begin{align*}
\sum_{j=0}^i(-1)^j\binom ij\J2(j+1)
&=3\zeta(2)\kakko{\binom{-\frac12}{i}^{\!\!2}-\binom{-\frac12}{i+1}^{\!\!2}}
=3\zeta(2)\binom{-\frac12}{i}^{\!\!2}\kakko{1-\frac{(2i+1)^2}{(2i+2)^2}},
\end{align*}
we have
\begin{align*}
\J4(n)=\sum_{p=0}^n (-1)^p\binom{-\frac12}p^{\!\!2}\binom np
\kakko{15\zeta(4)-3\zeta(2)\sum_{i=1}^{2p}\frac{(-1)^{i-1}}{i^2}}.
\end{align*}
\end{ex}

\subsection{Ascent operation and normalized higher Ap\'ery-like numbers}

For a given sequence $\{J(n)\}_{n\ge0}$, we associate a new sequence
\begin{equation}\label{eq:ascent_operation}
\ascent{J(n)}\deq
\sum_{p=0}^n (-1)^p\binom{-\frac12}p^{\!\!2}\binom np
\ckakko{\sum_{i=0}^{p-1}\frac{-1}{(2i+1)^2}\binom{-\frac12}{i}^{\!\!-2}
\sum_{j=0}^i(-1)^j\binom ijJ(j+1)}.
\end{equation}
Notice that $\ascent{J(0)}=0$.
It would be natural to extend $\ascent{J(n)}=0$ if $n<0$.
By the discussion in the previous subsection, we have the
\begin{lem}
Let $\{J(n)\}$ be a given sequence and $\{\ascent{J(n)}\}$ the one defined by \eqref{eq:ascent_operation}.
Then the equation
\begin{equation}
4n^2\ascent{J(n)}-(8n^2-8n+3)\ascent{J(n-1)}+4(n-1)^2\ascent{J(n-2)}=J(n)
\end{equation}
holds for $n\ge1$.
\end{lem}

Let us introduce the \emph{rational} sequences $\tJ k(n)$ by
\begin{align*}
\tJ1(n)&\deq\frac{2^n(n-1)!}{(2n-1)!!}\quad(n\ge1),\qquad
\tJ2(n)\deq\J2(n)/\J2(0)\quad(n\ge0),\\
\tJ k(n)&\deq\ascent{\tJ{k-2}(n)}\quad(k\ge3,\ n\ge0).
\end{align*}
We see that
\begin{equation}
\tJ{2k}(1)=\frac3{4^k},\quad
\tJ{2k+1}(1)=\frac2{4^k}.
\end{equation}
It is immediate to verify the
\begin{prop}
\begin{align}
\J k(n)=\sum_{m=0}^{\floor{k/2}-1}\zeta(k-2m,1/2)\tJ{2m+2}(n)%
+\frac{1-(-1)^k}2\tJ k(n).
\end{align}
\qed
\end{prop}
After this fact, we call $\tJ k(n)$ the \emph{normalized {\upshape(}higher{\upshape)} Ap\'ery-like numbers}.
By definition,
$\tJ k(n)$ for $k\ge2$ are written in the form
\begin{equation}
\tJ k(n)=\sum_{p=0}^n (-1)^p\binom{-\frac12}p^{\!\!2}\binom np\sS k(p),
\end{equation}
where
\begin{align*}
\sS2(p)&=1,\qquad
\sS3(p)=-2\sum_{i=0}^{p-1}\frac1{(2i+1)^3}\binom{-\frac12}i^{\!\!-2}
=-2\sum_{i=0}^{p-1}\frac{(1/2)_i(1)_i^3}{(3/2)_i^3}\frac{1^i}{i!},\\
\sS k(p)&=\sum_{i=0}^{p-1}\frac{-1}{(2i+1)^2}\binom{-\frac12}{i}^{\!\!-2}
\sum_{j=0}^i(-1)^j\binom ij\J{k-2}(j+1)\qquad(k\ge4).
\end{align*}

Thus it is enough to investigate $\sS k(p)$ to obtain an explicit expression for normalized Ap\'ery-like numbers.
\begin{lem}
\begin{equation}
\sS{k+2}(p+1)-\sS{k+2}(p)=\frac{\sS k(p+1)}{(2p+2)^2}-\frac{\sS k(p)}{(2p+1)^2}.
\end{equation}
\end{lem}

\begin{proof}
By definition, we have
\begin{equation}
\sS{k+2}(p+1)-\sS{k+2}(p)
=\frac{-1}{(2p+1)^2}\binom{-1/2}p^{\!\!-2}\sum_{j=0}^p(-1)^j\binom pj\tJ k(j+1).
\end{equation}
The sum in the right hand side is calculated as
\begin{align*}
&\eqsp\sum_{j=0}^p(-1)^j\binom pj\tJ k(j+1)\\
&=\sum_{j=0}^p(-1)^j\binom pj%
\kakko{\sum_{q=0}^j(-1)^q\binom{-1/2}q^{\!\!2}\binom{j+1}q\sS k(q)%
+(-1)^{j+1}\binom{-1/2}{j+1}^{\!\!2}\sS k(j+1)}\\
&=\sum_{q=0}^p(-1)^q\binom{-1/2}q^{\!\!2}\sS k(q)
\sum_{j=q}^p(-1)^j\binom pj\binom{j+1}q%
-\sum_{j=0}^p\binom pj\binom{-1/2}{j+1}^{\!\!2}\sS k(j+1).
\end{align*}
By the elementary identity
\begin{equation*}
\sum_{j=p}^n(-1)^j\binom nj\binom jp=(-1)^p\delta_{np}
\quad(n,p\in\Znn),
\end{equation*}
we get
\begin{align*}
\sum_{j=q}^p(-1)^j\binom pj\binom{j+1}q
&=\sum_{j=q}^p(-1)^j\binom pj\binom{j}q%
+\sum_{j=q}^p(-1)^j\binom pj\binom{j}{q-1}
=(-1)^q\kakko{\delta_{pq}+\binom p{q-1}}.
\end{align*}
Thus it follows that
\begin{align*}
\sum_{j=0}^p(-1)^j\binom pj\tJ k(j+1)
&=\binom{-1/2}p^{\!\!2}\sS k(p)%
+\sum_{q=0}^p\binom{-1/2}q^{\!\!2}\sS k(q)\binom p{q-1}%
-\sum_{j=0}^p\binom pj\binom{-1/2}{j+1}^{\!\!2}\sS k(j+1)\\
&=\binom{-1/2}p^{\!\!2}\sS k(p)%
-\binom{-1/2}{p+1}^{\!\!2}\sS k(p+1)\\
&=(2p+1)^2\binom{-1/2}p^{\!\!2}
\kakko{\frac{\sS k(p)}{(2p+1)^2}-\frac{\sS k(p+1)}{(2p+2)^2}}.
\end{align*}
Therefore we obtain
\begin{align*}
\sS{k+2}(p+1)-\sS{k+2}(p)
=\frac{\sS k(p+1)}{(2p+2)^2}-\frac{\sS k(p)}{(2p+1)^2}
\end{align*}
as we desired.
\end{proof}

As a corollary, we readily have the
\begin{lem}
\begin{equation}
\sS{k+2}(p)
=\sum_{q=1}^{p}\kakko{\frac{\sS k(q)}{(2q)^2}-\frac{\sS k(q-1)}{(2q-1)^2}}.
\end{equation}
\qed
\end{lem}
Using this lemma repeatedly, we obtain the
\begin{prop}
For each $r\ge1$,
\begin{align}
\sS{2r+2}(p)
&=\sum_{1\le i_1\le\dots\le i_r\le2p}\frac{(-1)^{i_1+\dots+i_r}}{i_1^2\dots i_r^2}
\e{i_1,\dots,i_r},\\
\sS{2r+3}(p)
&=\sum_{1\le2j-1<i_1\le\dots\le i_r\le2p}
\frac1{(2j-1)^3}\binom{-1/2}{j-1}^{\!\!-2}
\frac{(-1)^{i_1+\dots+i_r}}{i_1^2\dots i_r^2}\e{i_1,\dots,i_r},
\end{align}
where
\begin{equation}
\e{i_1,\dots,i_r}\deq
\begin{cases}
0 & 1\le\exists j<r\text{ s.t. }i_j=i_{j+1}\equiv1\pmod2,\\
1 & \text{otherwise}.
\end{cases}
\end{equation}
\qed
\end{prop}

\begin{ex}
We have
\begin{align*}
\sS4(p)&=\sum_{j=1}^{2p}\frac{(-1)^j}{j^2},\\
\sS6(p)&=\sum_{1\le i\le j\le 2p}\frac{(-1)^{i+j}}{i^2j^2}\e{i,j}
=\sum_{1\le i<j\le 2p}\frac{(-1)^{i+j}}{i^2j^2}+\sum_{i=1}^p\frac1{(2i)^4},\\
\sS8(p)&=\sum_{1\le i\le j\le k\le 2p}
\frac{(-1)^{i+j+k}}{i^2j^2k^2}\e{i,j,k}\\
&=\sum_{1\le i<j<k\le 2p}\frac{(-1)^{i+j+k}}{i^2j^2k^2}%
+\kakko{\sum_{1\le 2i\le 2p}\frac{1}{(2i)^4}}%
\kakko{\sum_{1\le k\le 2p}\frac{(-1)^k}{k^2}}.
\end{align*}
\end{ex}

\begin{rem}
We see that
\begin{equation}
\lim_{p\to\infty}\sS2(p)=1,\quad
\lim_{p\to\infty}\sS4(p)=-\frac{\pi^2}{12},\quad
\lim_{p\to\infty}\sS6(p)=-\frac{\pi^4}{720}.
\end{equation}
In general, we can prove that
\begin{equation}
\lim_{p\to\infty}\sS{2r+2}(p)=-\frac{\zeta(2r)}{2^{2r-1}}.
\end{equation}
See \cite{KY2007} for the proof.
\end{rem}

\section{Congruence relations among Ap\'ery-like numbers}

In this section, we study the congruence relation among
the \emph{normalized} Ap\'ery-like numbers introduced in the previous section.

\subsection{Congruence relations for Ap\'ery-like numbers}

We give several congruence relations among Ap\'ery-like numbers.

\begin{prop}[{\cite[Proposition 6.1]{KW2006a}}]
Let $p$ be a prime and $n=n_0+n_1p+\dots+n_kp^k$ be the $p$-ary expansion of $n\in\Znn$ $(0\le n_j<p)$.
Then it holds that
\begin{align*}
\tJ2(n)&\equiv\prod_{j=0}^k\tJ2(n_j) \pmod p.
\end{align*}
\qed
\end{prop}

The following claim is regarded as an analogue of Proposition \ref{prop:Apery_congruence}.
\begin{prop}[{\cite[Theorem 6.2]{KW2006a}}]
For any odd prime $p$ and positive integers $m,r$, the congruence relation
\begin{align*}
\tJ2(mp^r)&\equiv\tJ2(mp^{r-1})\pmod{p^r},\\
\tJ3(p^r)p^{3r}&\equiv\tJ3(p^{r-1})p^{3(r-1)}\pmod{p^r}.
\end{align*}
holds.
\qed
\end{prop}

\begin{prop}
For any odd prime $p$, the congruence relation
\begin{equation}
\sum_{n=0}^{p-1} \tJ2(n) \equiv 0 \pmod{p^2}
\end{equation}
holds.
\end{prop}

\begin{proof}
We see that
\begin{align*}
\sum_{n=0}^{p-1} \tJ2(n)
&=\sum_{n=0}^{p-1}\sum_{j=0}^n(-1)^j16^{-j}\binom{2j}j^{\!\!2}\binom nj
=\sum_{j=0}^{p-1}(-1)^j16^{-j}\binom{2j}j^{\!\!2}\sum_{n=j}^{p-1}\binom nj\\
&=\sum_{j=0}^{p-1}(-1)^j16^{-j}\binom{2j}j^{\!\!2}\binom p{j+1}
\equiv p\sum_{j=0}^{\frac{p-1}2}
16^{-j}\binom{2j}j^{\!\!2}\binom{p-1}j\frac{(-1)^j}{j+1}\\
&\equiv p\sum_{j=0}^{\frac{p-1}2}16^{-j}\binom{2j}j^{\!\!2}\frac1{j+1}
\pmod{p^2}
\end{align*}
since $\binom{2j}j^{\!2}$ is divisible by $p^2$ if $\frac{p-1}2<j<p$.
Notice that
\begin{equation*}
16^{-j}\binom{2j}j^{\!\!2}\equiv(-1)^j\binom{\frac{p-1}2+j}j\binom{\frac{p-1}2}j
\pmod p
\end{equation*}
for $0\le j<p$.
Hence we have
\begin{align*}
\sum_{n=0}^{p-1} \tJ2(n)
\equiv p\sum_{j=0}^{\frac{p-1}2}
\binom{\frac{p-1}2+j}j\binom{\frac{p-1}2}j\frac{(-1)^j}{j+1}
\pmod{p^2}.
\end{align*}
By putting $n=\frac{p-1}2$ and $m=0$ in the identity
(see \cite[Chapter 5.3]{GKP})
\begin{equation}
\sum_{k\ge0}\binom{n+k}k\binom nk\frac{(-1)^k}{k+1+m}
=(-1)^n\frac{m!n!}{(m+n+1)!}\binom mn,
\end{equation}
we have
\begin{equation}
\sum_{j=0}^{\frac{p-1}2}
\binom{\frac{p-1}2+j}j\binom{\frac{p-1}2}j\frac{(-1)^j}{j+1}
=\frac{(-1)^{\frac{p-1}2}}{({\frac{p+1}2})!}\binom0{\frac{p-1}2}
=0.
\end{equation}
Hence we obtain the desired conclusion.
\end{proof}

\begin{prop}
For each odd prime $p$, it holds that
\begin{align}\label{eq:Apery-Apery}
\tJ2(\tfrac{p-1}2)\equiv\apery2(\tfrac{p-1}2)\pmod{p^2}.
\end{align}
Here $\apery2(n)$ is the Ap\'ery number for $\zeta(2)$.
\end{prop}

\begin{proof}
It is elementary to check that
\begin{align*}
\binom{(p-1)/2}k&\equiv\binom{-1/2}k\ckakko{1-p\sum_{j=1}^k\frac1{2j-1}} \pmod{p^2},\\
\binom{(p-1)/2+k}k&\equiv(-1)^k\binom{-1/2}k\ckakko{1+p\sum_{j=1}^k\frac1{2j-1}} \pmod{p^2}
\end{align*}
for $k=0,1,\dots,(p-1)/2$.
Using these equations, we easily see that both $\apery2(\tfrac{p-1}2)$ and $\tJ2(\tfrac{p-1}2)$
are congruent to
\begin{align*}
\sum_{k=0}^{(p-1)/2}(-1)^k\binom{-1/2}k^{\!3}\ckakko{1-p\sum_{j=1}^k\frac1{2j-1}}
\end{align*}
modulo $p^2$.
\end{proof}

\begin{rem}
The following \emph{supercongruence}
\begin{align*}
\apery2(\tfrac{p-1}2)\equiv\lambda_p\pmod{p^2}
\end{align*}
holds if $p$ is a prime larger than $3$ (see \cite{I1990}; see also \cite{Mo2005, vH1987}).
\end{rem}

\subsection{Conjectures}

In the final position, we give several conjectures on congruence relations among normalized (higher) Ap\'ery-like numbers.

\begin{conj}[Remark 6.3 in \cite{KW2006a}]\label{conj:KW2006a_Rem_6.3}
For any odd prime $p$, the congruence relation
\begin{equation}
\sum_{n=0}^{p-1} \tJ2(n)^2 \equiv \LS{-1}p \pmod{p^3}
\end{equation}
holds.
\end{conj}

\begin{rem}
The conjecture above is quite similar to the Rodriguez-Villegas-type congruence due to Mortenson \cite{Mo2003}
\begin{align}\label{eq:RV-by-Mortenson}
\sum_{n=0}^{p-1}\binom{2n}n^{\!2}16^{-n}\equiv\LS{-4}p\pmod{p^2}.
\end{align}
We also remark that the following ``very similar'' congruence relation is obtained in \cite{OS}:
\begin{align}
\sum_{n=0}^{(p-1)/2}\binom{2n}n^{\!2}16^{-n}+\frac38p(-1)^{(p-1)/2}\sum_{i=1}^{(p-1)/2}\binom{2i}i\frac1i\equiv\LS{-1}p\pmod{p^3},
\end{align}
where $p$ is an arbitrary odd prime number.
\end{rem}

The following conjecture is regarded as a ``true'' analogue of Proposition \ref{prop:Apery_congruence}:
\begin{conj}[Kimoto-Osburn \cite{KO2008}]
For any odd prime $p$, the congruence relation
\begin{align}\label{eq:KO-conjecture}
\tJ2(mp^r-1)\equiv\LS{-1}p\tJ2(mp^{r-1}-1)\pmod{p^r}
\end{align}
holds for any $m,r\ge1$.
\end{conj}

\begin{rem}
When $r=1$,
the conjecture \eqref{eq:KO-conjecture} is obtained by using the elementary formulas
\begin{align*}
\binom{-1/2}{kp+j}\equiv\binom{-1/2}k\binom{-1/2}j\pmod{p},\qquad
\binom{mp-1}{n}\equiv(-1)^{n-\floor{n/p}}\binom{m-1}{\floor{n/p}}\pmod{p},
\end{align*}
and Mortenson's result \eqref{eq:RV-by-Mortenson}.
\end{rem}

\begin{conj}
For any odd prime $p$ and $m,r\in\Zp$ with $m$ odd, it holds that
\begin{align}\label{eq:ASD-type_congruence_for_J2}
\tJ2(\tfrac{mp^r-1}2)-\lambda_p\tJ2(\tfrac{mp^{r-1}-1}2)+(-1)^{p(p-1)/2}p^2\tJ2(\tfrac{mp^{r-2}-1}2)\equiv0 \pmod{p^r},
\end{align}
where $\lambda_n$ is given by
\begin{align*}
\sum_{n=1}^\infty \lambda_nq^n=q\prod_{n=1}^\infty(1-q^{4n})^6=\eta(4\tau)^6.
\end{align*}
Further, the congruence \eqref{eq:ASD-type_congruence_for_J2} holds modulo $p^{2r}$ if $p\ge5$.
\end{conj}

Notice that \eqref{eq:Apery-Apery} is a special case of the conjecture above (see \cite{Mo2005, vH1987}).
It is remarkable that both $\apery2(\tfrac{mp^r-1}2)$ and $\tJ2(\tfrac{mp^r-1}2)$ satisfy exactly the same congruence relation
(\eqref{eq:ASD-type_congruence_for_A2} and \eqref{eq:ASD-type_congruence_for_J2}),
though they are \emph{not} congruent modulo $p^r$ in general.

\begin{conj}
For any odd prime $p$, the congruence relation
\begin{equation}
\sum_{n=0}^{p-1} \tJ{2k}(n) \equiv -1 \pmod{p^2}
\end{equation}
holds for any $k\ge2$.
\end{conj}

\begin{ackn}
The author would like to thank the Institut des Hautes \'Etudes Scientifiques for the kind hospitality
during his stay in the fall 2008.
In fact, large part of the parer was written during this stay.
The author also thanks Robert Osburn for pointing out an error in the first draft
and telling him a right reference.
The author is partially supported by Grand-in-Aid for Young Scientists (B) No. 20740021.
\end{ackn}


\bigskip\noindent
Department of Mathematical Sciences,
University of the Ryukyus\\
Nishihara, Okinawa 903-0231 Japan\\
\texttt{kimoto@math.u-ryukyu.ac.jp}

\end{document}